\documentclass[12pt]{article}
\usepackage{amsmath}
\usepackage{amssymb}
\usepackage{amsfonts}
\usepackage{authblk}
\usepackage{tikz}

\usepackage{amsthm}

\newtheorem*{notation*}{Notation}
\newtheorem*{definition}{Definition}
\newtheorem{theorem}{Theorem}[section]

\newtheorem{lemma}[theorem]{Lemma}
\newtheorem{corollary}[theorem]{Corollary}
\newtheorem{conjecture}[theorem]{Conjecture}

\newtheorem*{remark*}{Remark}
\numberwithin{equation}{section}

\def \a {\alpha}
\def \b {\beta}
\def \R {\mathbb{R}}
\def \N {\mathbb{N}}

\def \L {\mathcal{L}}
\def \D {\Delta}

\tikzstyle{cir} = [draw, circle, minimum height= 20 mm]

\title{Many cliques with few edges and bounded maximum degree}
\author[1]{Debsoumya Chakraborti\thanks{This work was supported in part by the Institute for Basic Science (IBS-R029-C1)}}
\author[2]{Da Qi Chen \thanks{This material is based upon work supported by the Air Force Office of Scientific Research under award number FA9550-20-1-0080}}
\affil[1]{\small Discrete Mathematics Group, Institute for Basic Science (IBS), Daejeon,~South~Korea}
\affil[2]{\small Department of Mathematical Sciences, Carnegie Mellon University, Pittsburgh,~USA}
\affil[ ]{\small Email:
\texttt{debsoumya@ibs.re.kr},
\texttt{daqic@andrew.cmu.edu}}
\begin{document}
\maketitle
\begin{abstract}
Generalized Tur\'an problems have been a central topic of study in extremal combinatorics throughout the last few decades. One such problem, maximizing the number of cliques of a fixed order in a graph with fixed number of vertices and bounded maximum degree, was recently completely resolved by Chase. Kirsch and Radcliffe raised a natural variant of this problem where the number of edges is fixed instead of the number of vertices. In this paper, we determine the maximum number of cliques of a fixed order in a graph with fixed number of edges and bounded maximum degree, resolving a conjecture by Kirsch and Radcliffe. We also give a complete characterization of the extremal graphs. 
\end{abstract}

\section{Introduction}

In extremal graph theory, there are many recent works in the literature involving maximizing the number of complete subgraphs under certain natural conditions. One such old and classical result is a generalization of Tur\'an's theorem where Zykov \cite{Z} determined the maximum number of cliques of a fixed order in a graph with fixed number of vertices and bounded clique number (also see, e.g., \cite{E}, \cite{H}, \cite{R}, and \cite{S}). In this paper, we consider similar problems when the maximum degree is bounded instead of the clique number. On this topic, Cutler and Radcliffe \cite{CR} proved the following result answering a question of Galvin \cite{G} in a stronger form.

\begin{theorem} [\cite{CR}] \label{CR}
For any positive integers $n$ and $\D$, among all graphs on $n$ vertices with maximum degree $\D$, the graph $qK_{\D+1} \cup K_r$ ($q$ disjoint copies of $K_{\Delta+1}$ together with a single copy of $K_r$) uniquely maximizes the total number of complete subgraphs, where $n = q(\D+1) + r$, $0 \le r \le \D$. 
\end{theorem}

A natural generalization of the above theorem is to ask whether the statement is still true if we maximize the number of cliques of a fixed size $t \ge 3$ instead. Motivated by another related question from Engbers and Galvin \cite{EG}, Gan, Loh, and Sudakov proved a special case of this generalization where $q=1$ and $t=3$. They also showed that maximizing the number of cliques of order $t$ for general $t > 3$ can be reduced to the case when $t=3$. After substantial progress made in \cite{ACM}, \cite{AM}, \cite{CR1}, \cite{EG}, \cite{GLS}, and \cite{LM}, Chase \cite{C} finally resolved the problem by solving the triangle case, and thus proving the following theorem. For the convenience of writing, we denote the number of cliques of order $t$ in a graph $G$ by $k_t(G)$.

\begin{theorem} [\cite{C}, \cite{GLS}] \label{Z}
For any positive integers $n$, $\D$, $t \ge 3$, and any graph $G$ on $n$ vertices with maximum degree $\D$, we have that $k_t(G) \le k_t(qK_{\D+1} \cup K_r)$, where $n = q(\D+1) + r$, $0 \le r \le \D$. Moreover, $qK_{\D+1} \cup K_r$ is the unique graph satisfying the equality when $t \le r$.
\end{theorem}

Motivated by the edge analogue of Zykov's theorem (see, e.g., \cite{F} and \cite{FFK}), Kirsch and Radcliffe \cite{KR} started studying a natural variant of the problem in Theorem \ref{Z} where they fix the number of edges instead of vertices. This line of study is vastly motivated by the celebrated Kruskal-Katona theorem (see, \cite{Ka} and \cite{Kr}). In order to describe this, we introduce the notions of colex (colexicographic) order and colex graphs. Colex order on the finite subsets of the natural number set $\mathbb{N}$ is defined as the following: for $A, B \subseteq \mathbb{N}$, we have that $A < B$ if and only if $\max((A \setminus B) \cup (B \setminus A)) \in B$. The colex graph $L_m$ on $m$ edges is defined as the graph with the vertex set $\mathbb{N}$ and edges are the first $m$ sets of size $2$ in colex order. Note when $m=\binom{r}{2}+s$ where $0\le s< r$, then $L_m$ is the graph containing a clique of order $r$ and an additional vertex adjacent to $s$ vertices of the clique. Kruskal-Katona theorem implies the following: 

\begin{theorem} [\cite{Ka}, \cite{Kr}] \label{maxkt}
For any positive integers $t$, $m$, and any graph $G$ on $m$ edges, we have that $k_t(G) \le k_t(L_m)$. Moreover, $L_m$ is the unique graph satisfying the equality when $s \ge t-1$, where $m = \binom{r}{2} + s$ with $0 \le s < r$.
\end{theorem}

\begin{remark*}
For Theorem \ref{maxkt}, if $r \ge t$ and $s < t-1$, a graph $G$ satisfies the equality if and only if $G$ is an $m$-edge graph which contains $K_r$ as a subgraph. On the other hand, if $r < t$, then any graph with $m$ edges satisfies the equality, because no graph on $m$ edges has a copy of $K_t$.
\end{remark*}

We often use the following weaker version of Theorem \ref{maxkt}, which appears as Exercise 31b in Chapter 13 from Lov\'asz's book \cite{L}. From now on, the generalized binomial coefficient $\binom{x}{k}$ is defined to be the number $\frac{1}{k!} (x)(x-1)(x-2) \cdots (x-k+1)$, which exists for all real $x$.

\begin{corollary} \label{cor:maxkt}
Let $t \ge 3$ be an integer, and let $x \ge t$ be a real number. Then, every graph with exactly $\binom{x}{2}$ edges contains at most $\binom{x}{t}$ cliques of order $t$.
\end{corollary}

Kirsch and Radcliffe conjectured the following in \cite{KR} while studying the edge variant of the problem in Theorem \ref{Z}.

\begin{conjecture} [\cite{KR}] \label{conjecture}
For any $t \ge 3$, if $G$ is a graph with $m$ edges and maximum degree at most $\D$, then $k_t(G) \le k_t(qK_{\D+1} \cup L_b)$, where $m = q\binom{\D+1}{2} + b$ and $0 \le b < \binom{\D+1}{2}$. 
\end{conjecture}

Kirsch and Radcliffe \cite{KR} proved Conjecture \ref{conjecture} for $t=3$ and $\D \le 8$. They have also pointed out that unlike the vertex version \cite{GLS}, it might be difficult to reduce the general $t$ case to the $t=3$ case. Later, Kirsch and Radcliffe \cite{KR1} proved a weaker version of Conjecture \ref{conjecture}. For the convenience of writing this result, for a graph $G$, define $\tilde{k}(G) = \sum_{t \ge 2} k_t(G)$.

\begin{theorem} [\cite{KR1}] \label{KR1}
If $G$ is a graph with $m$ edges and maximum degree at most $\D$, then $\tilde{k}(G) \le \tilde{k}(qK_{\D+1} \cup L_b)$, where $m = q\binom{\D+1}{2} + b$ and $0 \le b < \binom{\D+1}{2}$. Moreover, $qK_{\D+1} \cup L_b$ is the unique graph satisfying the equality when $s \neq 1$, where $b = \binom{r}{2} + s$ with $0 \le s < r$. If $s=1$, then $qK_{\D+1} \cup K_r \cup K_2$ is the only extremal graph other than the one already described. 
\end{theorem}

In this paper, we resolve Conjecture \ref{conjecture} in a stronger form by characterizing the extremal graphs. By handling the general $t\ge 3$ case directly, we circumvented the aforementioned difficulty of reducing $t>3$ to the triangle case. Note that adding or deleting isolated vertices from a graph does not change the number of edges nor the number of $K_t$'s. Then, for the sake of this paper, two graphs are considered to be equivalent if they are isomorphic after deleting all the isolated vertices. In other words, whenever we talk about graphs in this paper, we assume it has minimum degree at least $1$. 

In order to characterize the extremal graphs, let us define a class of graphs. 

\begin{definition}
For $m=0$, let $\L_{t,\D}(m)$ be the family of empty graph, and for $0 < m \le \binom{\D+1}{2}$, call $\L_{t,\D}(m)$ to be the  following family of graphs, where $m = \binom{r}{2} + s$ with $0 \le s < r$.
\begin{itemize}
\item If $s \ge t-1$, then $\L_{t,\D}(m)$ contains just the colex graph $L_m$.
\item If $r \ge t$ and $s < t-1$, then $\L_{t,\D}(m)$ contains not only $L_m$, but also all $m$-edge graphs with maximum degree at most $\D$ that contain $K_r$ as a subgraph.
\item If $r < t$, then $\L_{t,\D}(m)$ contains not only $L_m$, but also all $m$-edge graphs with maximum degree at most $\D$.
\end{itemize} 
\end{definition}

Now, we state our main result. 

\begin{theorem} \label{main}
For any $3 \le t \le \D + 1$, if $G$ is a graph with $m$ edges and maximum degree at most $\D$, then $k_t(G) \le k_t(qK_{\D+1} \cup L_b)$, where $m = q\binom{\D+1}{2} + b$ and $0 \le b < \binom{\D+1}{2}$. Moreover, $G$ is an extremal graph if and only if $G$ is isomorphic to $qK_{\D+1} \cup L$ for some $L \in \L_{t,\D}(b)$. 
\end{theorem}

Note that if $t > \D + 1$ in Theorem \ref{main}, then any graph $G$ with maximum degree $\D$ cannot contain a copy of $K_t$, hence $k_t(G) = 0$. It can be easily checked that the case $q=0$ in Theorem \ref{main} is a direct corollary of Theorem \ref{maxkt}. Observe that since $k_t(L_{b_1}) \le k_t(L_{b_2})$ when $b_1 \le b_2$, the extremal number of $K_t$'s is non-decreasing in terms of $m$ in Theorem \ref{main}. We also remark that for the extremal construction in Theorem \ref{main}, the most restrictive case is when $s \ge t-1$, and we get a unique extremal graph. Note that as a corollary of Theorem \ref{main}, we can obtain Theorem \ref{KR1}.  

The rest of this paper is organized as follows. As the general proof for Theorem \ref{main} is quite cumbersome, we devote Section 2 for developing some structural deductions of a minimum counterexample if exists. We then use this structural information to give a short proof for the case $t=3$ in Section 3. We prove Theorem \ref{main} in its full generality in Section 4. In Section 5, we end with a few concluding remarks.

\section{Minimum counterexample}

From now on, we assume that $G$ is a minimum counterexample with respect to the number of edges to Theorem \ref{main}. In particular, we assume that $G$ has $q\binom{\Delta+1}{2}+\binom{r}{2}+s$ many edges where $0\le s < r\le \Delta$ (note that any non-negative integer $b$ can be uniquely written in the form $b = \binom{r}{2} + s$ for integers $0 \le s < r$), $k_t(G)\ge q\binom{\Delta+1}{t}+\binom{r}{t}+\binom{s}{t-1}$, and $G$ is not one of the extremal structures (i.e. $G\neq qK_{\Delta+1}\cup L$ where $L\in \mathcal{L}_{t, \Delta}(\binom{r}{2}+s)$). Furthermore, we assume that any other graph with the same number of edges has at most as many $K_t$'s as $G$. As we have discussed in the introduction, we may also assume that $q\ge 1$. In this section, we will show that $G$ must have certain structural properties which will be used in later sections to achieve contradictions.

\begin{lemma}
	\label{lem:posr}
	If $G$ is a minimum counterexample, then $r \ge 2$. 
\end{lemma}
	
\begin{proof}
	By the assumption that $0 \le s < r$, we have that $r \ge 1$. Suppose for the sake of contradiction that $r=1$. This implies that $s = 0$ and $|E(G)|=q\binom{\Delta+1}{2}$. For any edge $e$, its endpoints have at most $\D-1$ common neighbors in $G$. Hence, $e$ belongs in at most $\binom{\Delta-1}{t-2}$ distinct copies of $K_t$. Summing over all edges, since each $K_t$ is counted $\binom{t}{2}$ times, $k_t(G)$ is at most $q\binom{\Delta+1}{2}\frac{\binom{\Delta-1}{t-2}}{\binom{t}{2}}=q\binom{\Delta+1}{t}$. Then, $k_t(G)=q\binom{\Delta+1}{t}$ and every edge is in $\binom{\Delta-1}{t-1}$ copies of $K_t$. It is not hard to check that every edge must belong in a clique of order $\Delta+1$ and thus $G$ is a union of such cliques, contradicting the fact that $G$ is not an extremal structure.
\end{proof}


\begin{lemma}
	\label{alledget}
	If $G$ is a minimum counterexample, then every edge in $G$ is in a copy of $K_t$. 
\end{lemma}

\begin{proof}
	Suppose for the sake of contradiction that an edge $e$ is not in a copy of $K_t$. Consider the graph $G'=G\backslash e$. 
	\smallskip
	
	Case 1: $s>0$. Observe that by the minimality of $G$, we have that $q\binom{\D+1}{t}+\binom{r}{t}+\binom{s}{t-1}\le k_t(G) = k_t(G') \le q\binom{\D+1}{t}+\binom{r}{t}+\binom{s-1}{t-1}$. Since the right-hand-side is at most the left-hand-side, the above equation is at equality and $s<t-1$. Then, $G'$ is an extremal structure $qK_{\Delta+1}\cup L$ where $L\in \mathcal{L}_{t, \Delta}(\binom{r}{2}+s-1)$. Since $G$ has degree upper-bounded by $\Delta$, edge $e$ cannot be incident to any of the cliques $K_{\Delta+1}$. Note $L\cup\{e\}$ is a graph with maximum degree at most $\Delta$. If $t>r$, it follows that $L\cup\{e\}\in \mathcal{L}_{t, \Delta}(\binom{r}{2}+s)$ and if $t\le r$, then $L$ contains a copy of $K_r$ and $L\cup\{e\}\in \mathcal{L}_{t, \Delta}(\binom{r}{2}+s)$ as well. In both cases, $G$ is an extremal structure, a contradiction.
	\smallskip
	
	Case 2: $s=0$. Keep in mind that by Lemma \ref{lem:posr}, we have that $r-2\ge 0$. Then, $q\binom{\D+1}{t}+\binom{r}{t}\le k_t(G) = k_t(G') \le q\binom{\D+1}{t}+\binom{r-1}{t}+\binom{r-2}{t-1}$. Once again, the above equation is tight and thus $r<t$. Then, $G'$ is an extremal structure with $q$ cliques of order $\Delta+1$ and a graph $L\in \mathcal{L}_{t, \Delta}(\binom{r}{2}-1)$. Similarly, $e$ cannot be incident to any vertex in a clique of order $\Delta+1$. However, since $r<t$, $L \cup \{e\} \in \mathcal{L}_{t, \Delta}(\binom{r}{2})$. Thus $G$ is an extremal structure, a contradiction. 
	
\end{proof}

We next state an easy lemma about the binomial function which will become handy throughout the paper.

\begin{lemma} \label{easy:convex}
	Let $t$, $w$, $x$, $y$, and $z$ be non-negative integers such that $t \ge 2$, $x+w=y+z$, $x\ge y$, $x\ge z$, and $x\ge t$. Then, $\binom{x}{t}+\binom{w}{t} \ge \binom{y}{t}+\binom{z}{t}$. Moreover, the inequality is strict if $x> y$ and $x>z$. 
\end{lemma}

\begin{proof}
	Rearranging the inequality, we will show that $\binom{x}{t}-\binom{y}{t}\ge \binom{z}{t}-\binom{w}{t}$. Let $Y$ and $Z$ be two groups of people such that $|Y|=y$, $|Z|=z$, and $|Y\cap Z|=w$. Then, note that $|Y\cup Z|=x$. The left-hand-side can be viewed as the number of ways of choosing $t$ people from $Y\cup Z$ such that not all of them are from $Y$. The right-hand-side can be viewed as the number of ways of choosing $t$ people from $Z$ such that not all of them are from $Y\cap Z$. Observe that any group formed from the right-hand-side is also a group from the left-hand-side, thus the right-hand-side is at most the left-hand-side. To achieve the strict inequality, assume that $x-z=|Y \setminus Z|$ and $x-y=|Z\setminus Y|$ are strictly positive. Since $x=|Y\cup Z|\ge t$ and $t \ge 2$, there exists at least one group of size $t$ that contains some people from $Y\setminus Z$ and some people from $Z\setminus Y$. This group is counted by the left-hand-side but not by the right-hand-side, proving the strict inequality in Lemma \ref{easy:convex}.
	
\end{proof}

We next show that a minimum counterexample must be connected. Many of the ingredients of our proof can be found in \cite{KR} (e.g., Lemma 11 in \cite{KR}). However, we prove this in details to make this paper self-contained.

\begin{lemma}
	If $G$ is a minimum counterexample, then $G$ is connected.
	\label{lem:connectt}
\end{lemma}

\begin{proof}
	Suppose for the sake of contradiction that $G$ is not connected. First we show that $G$ does not contain a clique of order $\Delta+1$. Suppose $G$ does contain one such clique, note that due to the maximum degree condition, the clique is disjoint from the rest of the graph. Then, by removing this clique, we obtain a graph $G'$ with $(q-1)\binom{\Delta+1}{2}+\binom{r}{2}+s$ edges with at least $(q-1)\binom{\Delta+1}{t}+\binom{r}{t}+\binom{s}{t-1}$ number of $K_t$'s. Since $G'$ is not a counterexample, $G'$ has exactly the optimal number of $K_t$'s and thus is one of the extremal structures. However, adding the clique back implies that $G$ is also one of the extremal structures, a contradiction. 
	
	Now, suppose $G$ contains a proper subgraph $H$ that is a union of connected components of $G$ where $|E(H)| \ge \binom{\Delta+1}{2}$. Since $H$ is not a counterexample to Theorem \ref{main}, either $H$ contains strictly fewer $K_t$'s than an extremal structure with the same number of edges, or $H$ is an extremal structure. In the first case, replacing $H$ with one of the extremal structures results in a graph that strictly increases the number of $K_t$'s in $G$ while maintaining the same number of edges, creating a worse minimum counterexample, a contradiction. In the latter case, $H$ contains at least one copy of $K_{\Delta+1}$, contradicting our previous claim. Thus, we may assume that all proper subgraphs that are a union of connected components of $G$ have strictly fewer than $\binom{\Delta+1}{2}$ edges. 
	
	Let $G_1$ be a connected component of $G$ and $G_2=G\backslash G_1$. Note that if one of them is not an extremal structure, then by replacing it with an extremal structure, one strictly increases the number of $K_t$'s of $G$, achieving a similar contradiction as before. Since neither are counterexamples nor contain at least $\binom{\Delta+1}{2}$ edges, we may assume that $G_i\in \mathcal{L}_{t, \Delta}(\binom{r_i}{2}+s_i)$ where $|E(G_i)|=\binom{r_i}{2}+s_i$ and $0\le s_i< r_i\le \D$ for $i=1, 2$. Observe that if $0<s_i<t-1$, some edges in $G_i$ will not be part of any $K_t$'s, contradicting Lemma \ref{alledget}. Similar contradiction is achieved if $r_i<t$. Thus we may assume that $r_i\ge t$ and either $s_i=0$ or $s_i\ge t-1$ for $i=1, 2$. Observe that in either case, $G_i$ is the colex graph. Without loss of generality, we can assume that $r_1 \ge r_2$. Now, depending on the values of $s_1$ and $s_2$, we will move a certain amount of edges from $G_2$ to $G_1$ and obtain a graph with strictly more $K_t$'s than before, achieving a contradiction.
	\smallskip
	
	Case 1: $s_1, s_2> 0$. Let $s'=\min\{s_2, r_1-s_1\}$. Note that $s' \ge 1$. Let $G_1'$ be the colex graph with $|E(G_1)|+s'=\binom{r_1}{2}+s_1+s'$ edges and let $G_2'$ be the colex graph with $|E(G_2))|-s'=\binom{r_2}{2}+s_2-s'$ edges. Essentially, we are moving $s'$ edges from $G_2$ to $G_1$. The value $s'$ is chosen such that this process is equivalent to moving one edge at a time from $G_2$ to $G_1$ and stopping as soon as one of $G_1$ and $G_2$ becomes a clique. Doing so keeps the calculation of $k_t(G_1')$ and $k_t(G_2')$ as simple as possible. Note that $s_1, s_2 > 0$, and $s' \ge 1$. Hence, we have the conditions that $s_1, s_2 < s_1+s'$ and $s_1+s' \ge t-1$. Then, it follows from Lemma \ref{easy:convex} that
	
	\begin{align*}
		k_t(G_1\cup G_2) &= \binom{r_1}{t}+\binom{s_1}{t-1}+\binom{r_2}{t}+\binom{s_2}{t-1}\\
		&< \binom{r_1}{t} + \binom{s_1+s'}{t-1} + \binom{r_2}{t} + \binom{s_2-s'}{t-1}\\
		&= k_t(G_1'\cup G_2'), 
	\end{align*} 
	a contradiction.
	\smallskip
	
	Case 2: $s_1 = 0$, $s_2> 0$. Let $s'=\min\{r_2 -1, r_1-s_2\}$. Recall that $r_2\ge t\ge 2$ so after deleting $s_2$ edges from $G_2$, there are still more edges one can remove. Then, let $G_1'$ be the colex graph with $|E(G_1)|+s_2+s'=\binom{r_1}{2}+s_2+s'$ edges and let $G_2'$ be the colex graph with $|E(G_2)|-s_2-s'=\binom{r_2-1}{2} +r_2-1-s'$ edges (moving $s_2+s'$ edges from $G_2$ to $G_1$). Since $s'\ge r_2-1> s_2-1 \ge t-2> 0$, we have that $s_2+s'\ge t-1$. Furthermore, since $r_1\ge r_2$, we have that $s_2+s' = \min\{r_2-1+s_2, r_1\}> r_2-1, s_2$, Then, by Lemma \ref{easy:convex},
	
	\begin{align*}
		k_t(G_1\cup G_2) &= \binom{r_1}{t}+\binom{r_2-1}{t}+\binom{r_2-1}{t-1}+\binom{s_2}{t-1}\\
		&< \binom{r_1}{t}  +\binom{s_2+s'}{t-1} + \binom{r_2-1}{t} + \binom{r_2-1-s'}{t-1}\\
		&= k_t(G_1'\cup G_2'),
	\end{align*}
	a contradiction. 
	\smallskip
	
	Case 3: $s_1>0$, $s_2 = 0$. Let $s'=\min\{r_1-s_1, r_2-1\}$. Keep in mind that $r_1>s_1\ge t-1\ge 1$. Let $G_1'$ be the colex graph with $|E(G_1)|+s'=\binom{r_1}{2}+s_1+s'$ edges and let $G_2'$ be the colex graph with $|E(G_2)|-s'=\binom{r_2-1}{2}+r_2-1-s'$ edges (moving $s'$ edges from $G_2$ to $G_1$ while ensuring not adding more than $r_1-s_1$ many edges to $G_1$). Note that $s'\ge 1$ and $s_1+s'\ge t-1$. Since $r_1\ge r_2$, we have that $s_1+s'=\min\{r_1, r_2-1+s_1\}>r_2-1, s_1$. Then, it follows from Lemma \ref{easy:convex} that
	
	\begin{align*}
		k_t(G_1\cup G_2) &= \binom{r_1}{t}+\binom{s_1}{t-1}+\binom{r_2-1}{t}+\binom{r_2-1}{t-1}\\
		&< \binom{r_1}{t}+\binom{s_1+s'}{t-1} +\binom{r_2-1}{t}+\binom{r_2-1-s'}{t-1}\\
		&= k_t(G_1'\cup G_2'),
	\end{align*}
	a contradiction.
	\smallskip
	
	Case 4: $s_1=s_2=0$. Keep in mind that $r_2\ge t\ge 3$. Let $G_1'$ be the colex graph with $|E(G_1)|+r_2=\binom{r_1}{2}+r_2$ edges and let $G_2'$ be the colex graph with $|E(G_2)|-r_2=\binom{r_2-2}{2}+r_2-3$ edges. Then by Lemma \ref{easy:convex}, 
	
	\begin{align*}
		k_t(G_1\cup G_2) &= \binom{r_1}{t}+\binom{r_2-2}{t}+\binom{r_2-1}{t-1}+\binom{r_2-2}{t-1}\\
		&< \binom{r_1}{t}+\binom{r_2}{t-1} +\binom{r_2-2}{t}+\binom{r_2-3}{t-1}\\
		&= k_t(G_1'\cup G_2'),
	\end{align*}
	a contradiction.
	
\end{proof}

\begin{lemma}
	\label{lem:largen}
	If $G$ is a minimum counterexample, then $G$ has at least $q(\D+1) + (r+1)$ vertices.
\end{lemma}

\begin{proof}
	Suppose for the sake of contradiction that $G$ has $q'(\Delta+1)+r'<q(\Delta+1)+r+1$ vertices, where $r' \le \D$. It follows from Theorem \ref{Z} that $G$ has at most $q'\binom{\Delta+1}{t}+\binom{r'}{t}$ copies of $K_t$. If $s\ge t-1$, the number of $K_t$'s is strictly less than $k_t(qK_{\D+1} \cup L_{\binom{r}{2} + s})$, a contradiction. In fact, $k_t(G)$ is strictly less than $T_t := q\binom{\D+1}{t} + \binom{r}{t} + \binom{s}{t-1}$ unless $q'=q$, $r'=r$, and $0\le s\le t-2$. However, if $q'=q$, $r'=r$, and $0\le s\le t-2$, by Theorem \ref{Z}, $qK_{\Delta+1}\cup K_r$ is the only structure that can achieve $T_t$ many $K_t$'s, implying that $G$ is one of the extremal graphs, a contradiction. 
	
\end{proof}

\section{Many triangles with few edges}

In this section, we focus specifically on the case when $t=3$ in order to better demonstrate the ideas and techniques we use to prove Theorem \ref{main}. We continue with the assumption that $G$ is a minimum counterexample where $|E(G)|=q\binom{\Delta+1}{2}+\binom{r}{2}+s$ with $0 \le s < r \le \D$, $q\ge 1$, and $G$ is not one of the extremal structures described in Theorem \ref{main} but $G$ has at least $T_3:= q\binom{\Delta+1}{3}+\binom{r}{3}+\binom{s}{2}$ many triangles. 

Conceptually, the strategy is as follows. For a vertex $v$, the number of triangles it partakes in is at most $\binom{d(v)}{2}$, where $d(v)$ denotes the degree of $v$. However, every missing edge in the subgraph induced by its neighborhood also reduces this upper-bound on the number of triangles $v$ can be part of. Thus, we would like to upper-bound $\sum_{v \in V(G)} \binom{d(v)}{2}$ and lower-bound the total number of missing edges in all the neighborhoods. Note that every missing edge in a neighborhood creates an induced copy of $K_{1, 2}$. The next two lemmas are used to lower-bound such objects. 

\begin{lemma}
	\label{lem:cut}
	If $G$ is a minimum counterexample, then $G$ does not contain an edge cut of size at most $r-1$. In particular, $G$ has minimum degree at least $r$.
\end{lemma}

\begin{proof}

Let $B$ be a smallest edge cut in $G$, and suppose for contradiction, that $B$ has size $\b \le r-1$. Note that by Lemma \ref{lem:connectt}, we have that $\b \ge 1$. Consider the graph $G'=G\backslash B$. Then, $|E(G')|= q \binom{\D+1}{2} + \binom{r}{2} + s - \b$. Note that any triangle involving an edge of $B$ contains exactly two edges of $B$ and any pair of edges in $B$ belongs in at most one distinct triangle of $G$. Then, it follows that the number of triangles involving an edge of $B$ is at most $\binom{\b}{2}$. Since $G'$ is not a counterexample to Theorem \ref{main}, we can upper-bound the number of triangles in $G'$ based on the value of $\b$. We will also use Lemma \ref{easy:convex}. 
\smallskip

Case 1: $\b \le s$. We have the following:
\begin{align*}
k_3(G) &\ge q \binom{\D+1}{3} + \binom{r}{3} + \binom{s}{2}\\
&\ge q \binom{\D+1}{3} + \binom{r}{3} + \binom{s-\b}{2} + \binom{\b}{2}\\
&\ge k_3(G') + \binom{\b}{2} \ge k_3(G).
\end{align*}

Then, the equation is tight and $G'$ achieves the maximum possible number of triangles. Therefore $G'$ is an extremal structure. Since $q\ge 1$, $G'$ contains a clique of size $\Delta+1$. Due to the maximum degree condition, this clique is disjoint in $G$, contradicting Lemma \ref{lem:connectt}. 
\smallskip

Case 2: $\b > s$. Then, by Lemma \ref{easy:convex}:
\begin{align*}
k_3(G) &\ge q \binom{\D+1}{3} + \binom{r}{3} + \binom{s}{2}\\
&\ge q \binom{\D+1}{3} + \binom{r-1}{3} + \binom{r-1 - (\b-s)}{2} + \binom{\b}{2}\\
&\ge k_3(G') + \binom{\b}{2} \ge k_3(G).
\end{align*}

Similarly, the equation is tight. Then, $G'$ is an extremal structure and contains a clique $K_{\Delta+1}$. Then $G$ also contains this clique, contradicting Lemma \ref{lem:connectt}. 

\end{proof}

\begin{lemma} \label{K12}
If $G$ is a minimum counterexample, then $G$ contains at least $r^2$ copies of induced $K_{1,2}$.
\end{lemma}

\begin{proof}
	By Lemmas \ref{lem:posr} and \ref{lem:largen} and using $q \ge 1$, the number of vertices in $G$ is at least $\Delta + 2$. Thus for any vertex $u$, there exists at least one vertex $v\notin \{u\cup N(u)\}$, where $N(u)$ denotes the set of neighbors of $u$. Then, given a vertex $u$, let $B_u$ denote the set of edges in the cut between $u\cup N(u)$ and the rest of the graph $G$. Note that for every edge $vw\in B_u$, the vertices $u$, $v$, and $w$ induce a copy of $K_{1, 2}$. Summing over all the vertices $\sum_{u\in V(G)}|B_u|$ counts every induced copy of $K_{1, 2}$ exactly twice. Then, it follows from Lemma \ref{lem:largen} and Lemma \ref{lem:cut} that the number of induced copies of $K_{1, 2}$ is at least $\frac{1}{2} \sum_{u\in V(G)}|B_u|\ge \frac{1}{2} \cdot 2r \cdot r = r^2$. 
\end{proof}

\begin{remark*}
	Note that the above proof works verbatim for the general $t\ge 3$ case as long as a version of Lemma \ref{lem:cut} is also true for the general case. Keep this in mind for Section 4 since we will use Lemma \ref{K12} directly without proof after having proven a version of Lemma \ref{lem:cut} for the general case.   
\end{remark*}

Next, we prove a technical lemma about convex functions, which plays a central role in our proof of Theorem \ref{main}. In particular it is used to upper-bound $\sum_{v \in V(G)} \binom{d(v)}{3}$ and $\sum_{v \in V(G)} \binom{d(v)}{t}$ for the general case.

\begin{lemma} \label{convex}
	Let $f : \N \rightarrow \N$ be a convex function such that $f(0) = 0$. Let $n$, $D$, $r$, and $\D$ be positive integers where $r < \D$ and $nr \le D \le n \D$. Then, the maximum value of $\sum_{i=1}^k f(x_i)$ under the constraints $k \ge n$, $r \le x_i \le \D$, and $\sum_{i=1}^k x_i = D$ can be achieved by a solution where $k = n$ and all the $x_i$'s except for possibly one is either $r$ or $\D$. Moreover, $\sum_{i=1}^k f(x_i) \le a f(\D) + (n-a) f(r)$, for the real number $a$ that satisfies $a \D + (n-a) r = D$.
\end{lemma}

For our application, the $x_i$'s represent the degree sequence of our graph $G$, $f$ is the binomial function $\binom{x}{t}$, $r$ and $\Delta$ are respectively the minimum and maximum degree constraints, $n$ is a lower-bound on the number of vertices, and $D$ is the total sum of degrees. We remark that the first part of Lemma \ref{convex} can be seen as a consequence of Karamata's inequality \cite{K}. However, we provide a short proof without using it.

\begin{proof} [Proof of Lemma \ref{convex}]
	Consider a relaxed version of the above constraints where we require $r\le x_i\le \Delta$ for $1\le i\le n$ but for all other $i> n$, we only require $0< x_i\le \Delta$. Note that any solution under the original constraints is also a valid solution under the relaxed constraints. Then, it suffices to show that there exists a maximizer of the relaxed version that also satisfies the conclusions of the original problem.
	
	Consider any sequence $x_1, \ldots, x_k$ that maximizes $\sum_{i=1}^n x_i$ among the sequences that satisfy the new relaxed constraints and maximize our objective. Without loss of generality, we may assume that $x_1\ge x_2\ge \cdots \ge x_k>0$. For a convex function $f$, note that whenever $0<b\le b'$, $f(b) + f(b') \le f(b-c) + f(b'+c)$ for any positive number $c$ such that $b-c \ge 0$. If $k>n$, then $x_n<\Delta$, otherwise $D=\sum_{i=1}^kx_i > n\Delta$, a contradiction. Then, one can decrease $x_k$ and increase $x_n$ by $1$ so that $\sum_{i=1}^k f(x_i)$ does not decrease, but $\sum_{i=1}^n x_i$ increases. This is a contradiction to our choice of the sequence $x_1, \ldots, x_k$. Thus one can obtain an optimal solution where $k=n$ and $r \le x_i \le \Delta$ for all $i$. Now, if there exists $r < x_i \le x_j < \Delta$, one can decrease $x_i$ and increase $x_j$ by $1$ to not decrease the objective value. Hence, there is a maximizer that contains at most one $x_i$ that is not $r$ nor $\Delta$, proving the first part of our lemma.
	
	To prove the moreover part, it follows from previous arguments that at most one of the $x_i$'s is strictly between $r$ and $\D$. Then, without loss of generality we assume that $x_i=\Delta$ for all $1\le i\le j$, $r \le x_{j+1} \le \D$ and $x_i=r$ for all $j+1<i\le n$. Now, find the unique $0 \le \a \le 1$ such that $\a \D + (1-\a) r = x_{j+1}$, and apply Jensen's inequality to get $f(x_{j+1}) \le \a f(\D) + (1-\a) f(r)$. Our lemma follows immediately where $a=j+\alpha$.
	
\end{proof}

\begin{proof}[Proof of Theorem \ref{main} for t=3]
	Let $G$ be a minimum counterexample. For a vertex $v$, let $\mu(v)$ be the number of edges missing in its neighborhood. The number of triangles a vertex $v$ partakes in is exactly $\binom{d(v))}{2}-\mu(v)$. Then, the number of triangles in $G$ is exactly $\frac{1}{3} \sum_{v\in V(G)} \binom{d(v)}{2}-\frac{1}{3} \sum_{v\in V(G)} \mu(v)$. The second sum is exactly the number of induced copies of $K_{1, 2}$ which is at least $r^2$ by Lemma \ref{K12}. To upper-bound the first sum consider the function $f : \N \rightarrow \N$ given by $f(x)=\binom{x}{2}$. Let $D=\sum_{v\in V(G)}d(v)=2(q\binom{\Delta+1}{2}+\binom{r}{2}+s)$ be the sum of degrees and $n = q(\D+1) + (r+1)$ represents a lower-bound on $|V(G)|$ as shown in Lemma \ref{lem:largen}. It can be easily checked that $D < n \D$. Now, note that $D = \sum_{v\in V(G)}d(v) \ge n \cdot \min_{v \in V(G)} d(v) \ge nr$. Thus, we have that $\D > r$. Then, one can apply Lemma \ref{convex} to bound the first sum. In particular, one can check that the resulting number $a=q(\Delta+1)-2\cdot\frac{r-s}{\Delta-r}$ and $n-a= r+1+2\cdot\frac{r-s}{\Delta-r}$. Then the number of triangles in $G$ can be bounded in the following manner.

\begin{align*}
	& \frac{1}{3} \sum_{v\in V(G)} \binom{d(v)}{2}-\frac{1}{3} \sum_{v\in V(G)} \mu(v)\\
	 & \le \frac{1}{3}  \left(q(\Delta+1)-2\cdot\frac{r-s}{\Delta-r}\right)\binom{\Delta}{2}+\frac{1}{3}\left(r+1+2\cdot\frac{r-s}{\Delta-r}\right)\binom{r}{2} - \frac{r^2}{3} \\
	&= T_3+\binom{r}{2}-\binom{s}{2}-\frac{2}{3}\cdot\frac{r-s}{\Delta-r}\left(\binom{\Delta}{2}-\binom{r}{2}\right)-\frac{r^2}{3}\\
	&= T_3+\frac{1}{2}(r-s)(r+s-1)-\frac{1}{3}(r-s)(\Delta+r-1)-\frac{r^2}{3}\\
	&= T_3-\frac{1}{6}(r-s)(2\Delta-r-s+1)-\frac{1}{3}(r^2-s(r-s)).
\end{align*}

Focusing on the last line, since $0\le s< r< \Delta$, it follows that $r-s$, $2\Delta-r-s+1$, and $r^2-s(r-s)$ are all strictly positive. Then, $G$ has strictly less than $T_3$ many triangles, a contradiction. 

\end{proof}

\section{Maximizing the number of $K_t$'s}
We now adapt the methods used in the previous section to prove Theorem \ref{main} for general $t$. We assume that $G$ is a minimum counterexample where $|E(G)|=q\binom{\Delta+1}{2}+\binom{r}{2}+s$ with $0 \le s < r \le \D$, $q\ge 1$, and $G$ is not one of the extremal structures described in Theorem \ref{main} but $G$ has at least $T_t:= q\binom{\Delta+1}{t}+\binom{r}{t}+\binom{s}{t-1}$ many $K_t$'s. Let us start with an analog of Lemma \ref{lem:cut}. 

\begin{lemma}
	\label{lem:cutt}
	If $G$ is a minimum counterexample, then $G$ does not contain an edge cut of size at most $r-1$. In particular, every vertex must also have degree at least $r$.
\end{lemma}

\begin{proof}
	For the sake of contradiction, assume that $G$ contains a minimum cut $B$ of $\b$ edges where $0<\b<r$. Consider the graph $G\backslash B$. Let $k=k_t(G)-k_t(G\backslash B)$ be the number of $K_t$'s that contain at least one edge in $B$. We first show that $k \le \binom{\b}{t-1}$. To see this, consider $k':=$ the number of $(t-1)$-edge trees in the graph induced by the edges in $B$. First observe that $k'\le \binom{\b}{t-1}$ since every tree of size $t-1$ contains distinct $(t-1)$-subsets of $B$. Next, observe that $k\le k'$ because, by the minimality of $B$, every $K_t$ restricted to the edges of $B$ contains a unique tree of size $t-1$ that is counted once in $k'$. This proves our claim. 
	
	Since $G\backslash B$ is not a counterexample, $k_t(G\backslash B)\le q\binom{\Delta+1}{t}+k_t(L_{\binom{r}{2}+s-\b})$. If $\b\le s$, then by Lemma \ref{easy:convex},  $k_t(L_{\binom{r}{2}+s}) - k_t(L_{\binom{r}{2}+s-\b}) \ge \left(\binom{r}{t}+\binom{s}{t-1}\right)-\left(\binom{r}{t} + \binom{s-\b}{t-1}\right)\ge \binom{\b}{t-1}$. If $\b>s$, since $\b\le r-1$, then by Lemma \ref{easy:convex}, $k_t(L_{\binom{r}{2}+s}) - k_t(L_{\binom{r}{2}+s-\b}) \ge \left(\binom{r}{t}+\binom{s}{t-1}\right)-\left(\binom{r-1}{t} + \binom{r-1+s-\b}{t-1}\right)\ge \binom{\b}{t-1}$. Note that in either case, we have shown that $k_t(L_{\binom{r}{2}+s}) - k_t(L_{\binom{r}{2}+s-\b}) \ge \binom{\b}{t-1}$. Then, we have the following.
	\begin{align*}
	k_t(G\backslash B) &= k_t(G) - k \\
	&\ge q\binom{\Delta+1}{t} + k_t(L_{\binom{r}{2}+s}) - \binom{\b}{t-1} \\
	&\ge q\binom{\Delta+1}{t} + k_t(L_{\binom{r}{2}+s-\b}) \ge k_t(G\backslash B).
	\end{align*}
	
	Since $G\backslash B$ is not a counterexample, we can conclude that $G\backslash B$ is one of the extremal structures in Theorem \ref{main}. Hence, $G\backslash B$, and thus $G$, must contain a copy of $K_{\D+1}$. The maximum degree condition of $G$ forces this copy of $K_{\D+1}$ to be disconnected from the rest of the graph, which is a contradiction to Lemma \ref{lem:connectt}.
	
\end{proof}

The next few lemmas show that we must have that $\Delta>r>t-1$. 

\begin{lemma} \label{Delta>r}
	If $G$ is a minimum counterexample, then $\Delta > r$. 
\end{lemma}

\begin{proof}
	Suppose for the sake of contradiction that $\Delta=r$. By Lemma \ref{lem:cutt}, every vertex has degree $r=\Delta$. Since $r > s$, by Lemma \ref{lem:largen}, the number of edges is $|V(G)| \cdot \frac{\Delta}{2}\ge (q(\Delta+1)+r+1) \cdot \frac{\Delta}{2}>q\binom{\Delta+1}{2}+\binom{r}{2}+s=|E(G)|$, a contradiction. 
\end{proof}

From now on, we may assume that $\Delta> r$. Note that each vertex partakes in at most $\binom{d(v)}{t-1}$ many copies of $K_t$. Then, a crude upper-bound on $k_t(G)$ is $\frac{1}{t}\sum_{v\in V(G)}\binom{d(v)}{t-1}$. Thus, we use Lemma \ref{convex} to bound the above sum. 

\begin{lemma}
	\label{lem:betterbound}
	If $G$ is a minimum counterexample, then $\frac{1}{t}\sum_{v\in V(G)}\binom{d(v)}{t-1}$ is at most
		\begin{itemize}
			\item $T_t+\binom{r}{t-1}-\binom{s}{t-1}-\frac{2(r-s)}{t}\binom{r}{t-2}$, if $r\ge t-1$,
			\item $T_t-\frac{2(r-s)}{t(\Delta-r)}\binom{\Delta}{t-1}$, if $r<t-1$. 
		\end{itemize}
\end{lemma}
	
\begin{proof}
	Consider the function $f : \N \rightarrow \N$ given by $f(x)=\binom{x}{t-1}$. Note that $f(x)$ is convex. Let $D=\sum_{v \in V(G)}d(v)=2(q\binom{\Delta+1}{2}+\binom{r}{2}+s)$ and $n=q(\Delta+1)+r+1$. By Lemma \ref{lem:cutt}, we have that $r\le d(v)\le \Delta$ and thus $nr\le D\le n\Delta$.  Then, one can apply Lemma \ref{convex} to bound $\frac{1}{t}\sum_{v\in V(G)}\binom{d(v)}{t-1}$. Since $\D > r$ (by Lemma \ref{Delta>r}), one can check that the value of $a$ in Lemma \ref{convex} is $q(\Delta+1)-\frac{2(r-s)}{\Delta-r}$. Then,  
	\begin{align*}
		\frac{1}{t}\sum_{v\in V(G)}\binom{d(v)}{t-1} &\le \frac{1}{t} \left(q(\Delta+1)-\frac{2(r-s)}{\Delta-r}\right)\binom{\Delta}{t-1} + \frac{1}{t}\left((r+1)+\frac{2(r-s)}{\Delta-r}\right)\binom{r}{t-1} \\
		&= T_t + \binom{r}{t-1} -\binom{s}{t-1} - \frac{2(r-s)}{t(\Delta-r)} \left(\binom{\Delta}{t-1}-\binom{r}{t-1}\right).
	\end{align*}

	When $r<t-1$, our lemma follows immediately from the last line.  When $r\ge t-1$, it suffices to prove $\binom{\Delta}{t-1}\ge \binom{r}{t-1}+ (\Delta-r)\binom{r}{t-2}$. The inequality follows by viewing the left-hand-side as forming a team of $t-1$ from a group of $\Delta$ people. The right-hand-side represents either choosing $t-1$ from a special subgroup of $r$ people or forming a team with one person outside the special subgroup while filling the rest with people from within the special subgroup. 
\end{proof}

\begin{corollary} \label{valueofr}
	If $G$ is a minimum counterexample, then $r\ge t$.
\end{corollary}

\begin{proof}
	Suppose for the sake of contradiction that $r\le t-1$. We can bound $k_t(G)$ by $\frac{1}{t} \sum_{v \in V(G)} \binom{d(v)}{t-1}$. Using the previous lemma, if $r<t-1$, it follows immediately that $k_t(G)<T_t$, a contradiction. If $r=t-1$, then $k_t(G)\le T_t+1-\frac{2}{t}(t-1)<T_t$, also a contradiction. 
\end{proof}

	From now on, we may assume that $\Delta>r>t-1$. Note that when $r\ge t$, we can no longer use the above lemma to conclude immediately that $k_t(G)< T_t$. This is because $\binom{d(v)}{t}$ is too crude of a bound for the number of $K_t$'s containing a particular vertex. There may be in fact many $(t-1)$-subsets in the neighborhood of $v$ that contain missing edges and cannot form a $K_t$. Therefore, our next goal is to lower-bound such objects. 
	
	On a high level, given a vertex $v$, we want to bound the number of incomplete $K_{t-1}$'s by the number of missing edges in its neighborhood. It would be ideal if each missing edge in $N(v)$ produced a lot of incomplete $K_{t-1}$'s and each incomplete $K_{t-1}$ did not involve too many missing edges. In some sense, the worst scenario is when all the incomplete $K_{t-1}$'s involve a lot of missing edges (i.e., a neighborhood where all the edges are missing). To show that the above worst case does not happen, we will prove that each neighborhood contain enough edges so that there exist a lot of $K_{t-1}$'s that do not have too many missing edges.

\begin{lemma}
	\label{lem:mindegt}
	If $G$ is a minimum counterexample, then the subgraph induced by the neighborhood of every vertex contains more than $\binom{r-1}{2}$ edges.  
\end{lemma}

\begin{proof}
	First, from Corollary \ref{valueofr}, we know that $r \ge t$. By applying Corollary \ref{cor:maxkt}, it suffices to show that every vertex of $G$ is in more than $\binom{r-1}{t-1}$ copies of $K_t$. Suppose for the sake of contradiction that there exists a vertex $v$ that is in at most $\binom{r-1}{t-1}$ copies of $K_t$. Let $\binom{x}{t-1}$ be the number of copies of $K_t$ containing $v$ for some real number $x \ge t-1$. By assumption, $x \le r-1$. It is also easy to see that $d(v) \ge x$. Consider the graph $G'=G\backslash v$.
	\smallskip
	
	Case 1: $x\le s$. Note that $G'$ contains at most $|E(G)| - \lceil x\rceil = q\binom{\D+1}{2} + \binom{r}{2} + s - \lceil x\rceil$ edges. Since $G'$ is not a counterexample, we have
	\begin{align*}
		k_t(G) &\ge q\binom{\Delta+1}{t}+\binom{r}{t}+\binom{s}{t-1} \\
		&\ge q\binom{\Delta+1}{t}+\binom{r}{t}+\binom{ s-\lceil x\rceil}{t-1}+\binom{\lceil x\rceil}{t-1} \\
		&\ge k_t(G')+\binom{x}{t-1} \ge k_t(G). 	
	\end{align*}	 
	
	Then it follows that the above equation is tight. Note that $t-1\le x \le s$. Since $s\ge t-1$, looking at the second inequality, by Lemma \ref{easy:convex}, $\lceil x\rceil=s$. Since the third inequality is tight, $x=\lceil x\rceil  =s$ and $G'$ has exactly $q\binom{\Delta+1}{t}+\binom{r}{t}$ copies of $K_t$. Since $G'$ is not a counterexample, $G'$ must contain exactly $q\binom{\D+1}{2} + \binom{r}{2}$ edges, and $G'$ is an extremal structure containing a copy of $K_{\D+1}$. Then, $G$ also contains a copy of $K_{\D+1}$ which must be disconnected from the rest of $G$ due to the maximum degree condition on $G$, contradicting Lemma \ref{lem:connectt}. 
	\smallskip
	
	Case 2: $x> s$. Note that $G'$ contains at most $|E(G)| - \lceil x\rceil = q\binom{\D+1}{2} + \binom{r-1}{2} + r - 1 + s - \lceil x\rceil$ edges. Since $G'$ is not a counterexample, we have
	\begin{align*}
		k_t(G) &\ge q\binom{\Delta+1}{t}+\binom{r-1}{t}+\binom{r-1}{t-1}+\binom{s}{t-1} \\
		&\ge q\binom{\Delta+1}{t}+\binom{r-1}{t} + \binom{ r-1+s-\lceil x\rceil}{t-1}+\binom{\lceil x \rceil}{t-1} \\
		&\ge k_t(G')+ \binom{x}{t-1} \ge k_t(G).	
	\end{align*}
	
	Thus the above inequalities are tight. Since the second inequality is tight and $r\ge t$, by Lemma \ref{easy:convex}, $\lceil x\rceil=r-1$. Since the third inequality is tight, $x=\lceil x\rceil =r-1$ and $G'$ has exactly $q\binom{\Delta+1}{t}+\binom{r-1}{t} + \binom{s}{t-1}$ copies of $K_t$. Furthermore, $G'$ contain exactly $q\binom{\D+1}{2} + \binom{r-1}{2} + s$ edges, and $G'$ is an extremal structure containing a copy of $K_{\D+1}$. Similar to Case 1, we get a contradiction to Lemma \ref{lem:connectt}, proving our lemma.
	
\end{proof}

	To aid the upcoming analysis, we prove that the function mapping $\binom{x}{2}$ to $\binom{x}{t-1}$ is convex, which was also useful in \cite{GLS}. 
	
\begin{lemma} \label{gconv}
	Let $u(x)$ be the positive root of $\binom{u}{2} = x$, i.e., $u(x) = \frac{1+\sqrt{1+8x}}{2}$. Then, the following function $g : \R \rightarrow \R$ is convex.
	\begin{align*}
	g(x) = 
		\begin{cases}
		\binom{u(x)}{t-1}, &\text{if	} x \ge \binom{t-2}{2} \\ 
		0, &\text{otherwise}
		\end{cases} 
	\end{align*}
\end{lemma}

\begin{proof}
	Observe that when $t=3$, $g$ is trivially convex. To prove convexity of $g$ for $t>3$, consider the derivative of $g$. Note that $g'(x)=0$ for all $x< \binom{t-2}{2}$. We will show that $g'(x)$ is non-negative and increasing for $x\ge \binom{t-2}{2}$. For the convenience of notation, let $u=u(x)$. Observe that $u'= \frac{2}{\sqrt{1+8x}}$ and 

	$$g'(x) = \frac{u'}{(t-1)!} \sum_{i=0}^{t-2}\frac{(u)_{t-1}}{u-i}.$$

	Note that $u'\cdot (u-i)=1-\frac{2i-1}{\sqrt{1+8x}}$ which is a non-negative non-decreasing function for $i= 1$ and $i= 2$, for $x\ge 1$. For $1\le i\le t-2$, the function $\frac{(u)_{t-1}}{u-i} = u(u-1)\cdots (u-i+1)(u-i-1)\cdots (u-t+2)$ is simply a product of $t-2$ non-negative linear terms that are also non-decreasing. Furthermore, $\frac{(u)_{t-1}}{u-i}$ contains a factor of $(u-1)$ or $(u-2)$, or possibly both. By selecting one of them to combine with $u'$, every term in the sum can be broken into a product of non-negative non-decreasing functions. Then, $g'(x)$ is  also non-negative and non-decreasing. Thus it follows that $g(x)$ is convex, finishing the proof for Lemma \ref{gconv}.
\end{proof}

\begin{corollary} \label{identity}
	Given integers $r, t$ where $r> t-1\ge 2$, the following inequality is true for all real numbers $x\ge r$:
	\begin{equation*}
	\frac{\binom{x}{t-1}-\binom{r-1}{t-1}}{\binom{x}{2}-\binom{r-1}{2}}\ge \frac{\binom{r-1}{t-2}}{r-1}.
	\end{equation*}
\end{corollary}

\begin{proof}
	Note that $\frac{g(c)-g(a)}{c-a} \ge \frac{g(b)-g(a)}{b-a}$ holds for any convex function $g$ and $a < b \le c$. To obtain Corollary \ref{identity}, one can use the function $g$ defined in Lemma \ref{gconv} and substitute $\binom{r-1}{2}$, $\binom{r}{2}$, and $\binom{x}{2}$ into $a$, $b$, and $c$.
\end{proof}

Define $\mu_t(v)$ to be the number of collections of $t-1$ neighbors of $v$ that do not induce a copy of $K_{t-1}$. We will use the above two lemmas to show a relationship between the number of missing edges in a neighborhood of $v$ and $\mu_t(v)$.

\begin{lemma} \label{count}
Let $G$ be a minimum counterexample. For a vertex $v$ in $G$, if the neighborhood $N(v)$ is missing $\mu = \mu(v)$ edges, then $\mu_t(v) \ge \mu \cdot \frac{\binom{r-1}{t-2}}{r-1}$.
\end{lemma}

\begin{proof}
Let $v\in V(G)$ and $H$ be the subgraph induced by $N(v)$. We will split the proof into two cases.
\smallskip

Case 1: $\mu \le r-1$. By Theorem \ref{maxkt}, the number of copies of $K_{t-1}$ in $H$ is at most as many as a colex graph with $|E(H)|=\binom{d(v)}{2}-\mu=\binom{d(v)-1}{2}+d(v)-1-\mu$ many edges. Note that by Lemma \ref{lem:cutt}, $d(v)-1 \ge r-1\ge \mu$. Then, $\mu_t(v) \ge \binom{d(v)}{t-1} - \left(\binom{d(v)-1}{t-1}+\binom{d(v)-1 - \mu}{t-2}\right) =\binom{d(v)-1}{t-2}-\binom{d(v)-1-\mu}{t-2}$. It follows from Lemma \ref{easy:convex} that $\mu_t(v)\ge \binom{r-1}{t-2} - \binom{r -1- \mu}{t-2}$. Then, the lemma follows immediately if we can prove the following claim. 
$$\binom{r-1}{t-2} - \binom{r-1 - \mu}{t-2} \ge \mu \cdot \frac{\binom{r-1}{t-2}}{r-1}.$$

Rearranging the terms, we can obtain the following equivalent form:

$$(r-1-\mu)\binom{r-1}{t-2}\ge (r-1)\binom{r-1-\mu}{t-2},$$

which is equivalent to

$$(r-2)(r-3)\cdots(r-t+2)\ge (r-2-\mu)(r-3-\mu)\cdots(r-t+2-\mu).$$

The above inequality is clearly true, thus proving the case when $\mu \le r-1$.
\smallskip

Case 2: $\mu \ge r$. Let $z$ be a positive real number such that $|E(H)|=\binom{z}{2}$. From Lemma \ref{lem:mindegt}, we get that $|E(H)| \ge \binom{r-1}{2}$ and hence $z \ge r-1$. An application of Corollary \ref{cor:maxkt} shows that $k_{t-1}(H) \le \binom{z}{t-1}$, hence $\mu_t(v) \ge \binom{d(v)}{t-1} - \binom{z}{t-1}$. Note that $\mu = \binom{d(v)}{2} - \binom{z}{2}$. Find positive $x$ such that $\mu = \binom{d(v)}{2} - \binom{z}{2} = \binom{x}{2} - \binom{r-1}{2}$, i.e. 
\begin{equation} \label{sum}
\binom{d(v)}{2} + \binom{r-1}{2} = \binom{x}{2} + \binom{z}{2}. 
\end{equation} 

Note for a convex function $g : \R \rightarrow \R$ and $a+b=c+d$ with $a \ge \max(c,d)$, we have that $g(a) + g(b) \ge g(c) + g(d)$. Hence, by Lemma \ref{gconv} and Equation \eqref{sum}, we obtain the following:
\begin{equation*} 
\binom{d(v)}{t-1} + \binom{r-1}{t-1} \ge \binom{x}{t-1} + \binom{z}{t-1}.
\end{equation*} 

Then,

$$\mu_t(v) \ge \binom{d(v)}{t-1} - \binom{z}{t-1} \ge \binom{x}{t-1} - \binom{r-1}{t-1} = \mu \cdot \frac{\binom{x}{t-1} - \binom{r-1}{t-1}}{\binom{x}{2} - \binom{r-1}{2}}.$$

Note that $\binom{x}{2}=\binom{r-1}{2}+\mu \ge \binom{r-1}{2}+r-1=\binom{r}{2}$, and thus $x\ge r$. Then, applying Corollary \ref{identity} finishes this proof for Lemma \ref{count}.
\end{proof}

\begin{corollary} \label{bound2}
Let $G$ be a minimum counterexample.Then 
$$\sum_{v \in V(G)} \mu_t(v) > r \cdot \binom{r-1}{t-2}.$$
\end{corollary}

\begin{proof}
	From Lemma \ref{lem:cutt}, we know that every edge cut in $G$ has size at least $r$. Then, using the exact same proof as Lemma \ref{K12}, we obtain a similar result where $G$ contains at least $r^2$ copies of $K_{1, 2}$. Recall that $\mu(v)$ denotes the number of edges missing in the graph induced by $N(v)$. It follows that $\sum_{v \in V(G)}\mu(v)\ge r^2$. Then, by Lemma \ref{count}, 
	\begin{align*}
		\sum_{v \in V(G)} \mu_t(v) \ge \sum_{v \in V(G)} \mu(v)\frac{\binom{r-1}{t-2}}{r-1} 
		\ge r^2 \cdot \frac{\binom{r-1}{t-2}}{r-1} > r \cdot \binom{r-1}{t-2}.
	\end{align*}
\end{proof}

Finally, we are ready to prove our main theorem. 

\begin{proof} [Proof of Theorem \ref{main}]
	Let $G$ be a minimum counterexample. By the definition of $\mu_t(v)$, the number of copies of $K_t$ containing a vertex $v$ is $\binom{d(v)}{t-1} - \mu_t(v)$. Then, the number of $K_t$'s in $G$ is $\frac{1}{t} \sum_{v \in V(G)} \binom{d(v)}{t-1} - \frac{1}{t} \sum_{v \in V(G)} \mu_t(v)$. We now use Lemma \ref{lem:betterbound} and Corollary \ref{bound2} to upper-bound the first and second sum respectively. Then,
	\begin{align}
		k_t(G) = & \frac{1}{t} \sum_{v \in V(G)} \binom{d(v)}{t-1} - \frac{1}{t} \sum_{v \in V(G)} \mu_t(v) \nonumber \\
		&\le T_t + \binom{r}{t-1} - \binom{s}{t-1} - \frac{2(r-s)}{t} \binom{r}{t-2} - \frac{r}{t} \binom{r-1}{t-2} \nonumber \\
		&= T_t + \frac{1}{t} \binom{r}{t-1} - \binom{s}{t-1} - \frac{2(r-s)}{t} \binom{r}{t-2} \label{ex}.
	\end{align}
	
	Recall that from Corollary \ref{valueofr} we may assume that $r \ge t$. If $s \le t-2$, then from Equation \eqref{ex} we have the following:
	\begin{align*}
	k_t(G) &\le T_t + \frac{r-t+2}{t(t-1)}\binom{r}{t-2}-\frac{2(r-t+2)}{t}\binom{r}{t-2} \\
	&\le T_t + \left(\frac{1}{t-1} - 2\right) \frac{r-t+2}{t(t-1)}\binom{r}{t-2} \\
	&< T_t,
	\end{align*}
	a contradiction.
	
	To achieve a similar contradiction for $s \ge t-2$, for a fixed positive integer $r$, consider the following function on $s$ when $t-2 \le s \le r-1$:
	
	$$h_r(s) = \frac{1}{t} \binom{r}{t-1} - \binom{s}{t-1} - \frac{2(r-s)}{t} \binom{r}{t-2}.$$

	Thus, to finish the proof of Theorem \ref{main}, it suffices to show that $h_r(s) < 0$ for all $t-2 \le s \le r-1$. Let $t-2 \le z\le r-1$ be a value of $s$ where $h_r(s)$ is maximum. Note that one of the following is true, $z = t-2$, $z=r-1$ or $h_r'(z)=0$. We will show that in all three cases, $h_r(z)<0$. 
	\smallskip
	
	Case 1: $z = t-2$. Then 
	
	$$h_r(t-2) \le \frac{r-t+2}{t(t-1)}\binom{r}{t-2}-\frac{2(r-t+2)}{t}\binom{r}{t-2}.$$
	
	Since $\frac{1}{t-1}<2$, $h_r(t-2)<0$. 
	\smallskip
	
	Case 2: $z=r-1$. Then 
	
	$$h_r(r-1) = \frac{1}{t} \left(\binom{r-1}{t-1}+\binom{r-1}{t-2}\right) - \binom{r-1}{t-1} - \frac{2}{t} \binom{r}{t-2}.$$
	
	Comparing the two positive terms with the two negative terms, $h_r(r-1)<0$. 
	\smallskip
	
	Case 3: $h_r'(z)=0$ and $z > t-2$. Then, the derivative of $h_r(s)$ with respect to $s$ is:
	
	$$h'_r(s) = - \frac{(s)_{t-1}}{(t-1)!} \sum_{i=0}^{t-2} \frac{1}{s-i} + \frac{2}{t} \binom{r}{t-2}.$$

	Since $h_r'(z)=0$, we have
	\begin{align*}
		\frac{2}{t} \binom{r}{t-2} & = \frac{(z)_{t-1}}{(t-1)!} \sum_{i=0}^{t-2} \frac{1}{z-i}\\
		& \le \frac{(z)_{t-1}}{(t-1)!}\sum_{i=0}^{t-2}\frac{1}{z-t+2} = \frac{(z)_{t-2}}{(t-2)!}.
	\end{align*}
	
	 Using the above bound, we have
	\begin{align*}
		h_r(z) & = \frac{1}{t} \binom{r}{t-1} - \binom{z}{t-1} - \frac{2(r-z)}{t} \binom{r}{t-2} \\
		&= \frac{r-t+2}{t(t-1)} \binom{r}{t-2} - \frac{z-t+2}{t-1} \cdot \frac{(z)_{t-2}}{(t-2)!} - \frac{2(r-z)}{t} \binom{r}{t-2} \\
		&\le \frac{r-t+2}{t(t-1)}\binom{r}{t-2} - \frac{z-t+2}{t-1} \cdot \frac{2}{t} \binom{r}{t-2} - \frac{2(r-z)}{t} \binom{r}{t-2} \\
		&= \frac{\binom{r}{t-2}}{t(t-1)} \left((r-t+2)-2(z-t+2)-2(r-z)(t-1)\right) \\
		&= \frac{\binom{r}{t-2}}{t(t-1)} \left(-(2t-3)(r-1-z) -z-(t-1)\right) < 0.
	\end{align*}

This concludes the proof of Theorem \ref{main}.

\end{proof}

\section{Concluding remarks}

The celebrated Erd\H{o}s-S\'os conjecture (see, eg., \cite{FS}) asks for the maximum number of edges in a graph with a fixed number of vertices and which does not contain a copy of a fixed tree $T$. A natural generalization is to maximize the number of cliques of a fixed size instead of maximizing the number of edges. Theorem \ref{Z}, for example, is the special case where we maximize the number of $K_t$'s while avoiding the star $K_{1, \Delta}$. In this paper, we answered the variation where the number of edges is fixed instead of the number of vertices. In this edge variant, we can similarly change the structure we avoid to any tree $T$. We believe that similar structures should still be extremal. However, these types of problems might be even more difficult than the vertex version.

It will be interesting to consider the following general version of Theorem \ref{main}. For $0 < s < t$, $m$ and $\D$ positive integers, determine the maximum number of $K_t$'s in a graph with $m$ copies of $K_s$ and whose maximum degree is upper-bounded by $\D$. We finish by remarking that if a reduction from general $t>3$ to the triangle case for Theorem \ref{main} exists, then the proof can potentially be considerably simplified. 

\section{Acknowledgement}
We thank Po-Shen Loh and Stijn Cambie for their careful proofreading and helping us to improve the writing. We are grateful to the anonymous referees for suggesting us ways to increase the quality of this paper.

\end{document}